\documentclass{amsart}
\usepackage[utf8]{inputenc}
\usepackage[utf8]{inputenc}
\usepackage{hyperref}
\usepackage{pst-node}
\usepackage{auto-pst-pdf}
\usepackage{tikz-cd,array} 
\usepackage{comment}
\usepackage{amsmath,amssymb,amsthm,stmaryrd,wasysym}
\usepackage{amsfonts,tikz}
\usepackage[margin=1in]{geometry}
\usepackage{mathtools}
\usepackage{youngtab}
\usepackage{ytableau}
\usepackage{pbox}
\usetikzlibrary{positioning,automata}
\tikzset{every loop/.style={min distance=10 mm, in=60, out=120, looseness=10}}
\tikzset{
dot/.style = {circle, fill, inner sep=2.2,outer sep=0},
}

\usepackage[normalem]{ulem}

\definecolor{darkblue}{rgb}{0.0,0,0.7}

\newcommand{\newword}[1]{\textcolor{darkblue}{\textbf{\textit{#1}}}}

\newcommand{\bk}{\backslash}
\newcommand{\lm}{\lambda}

\newcommand{\Lm}{\Lambda}
\newcommand{\tm}{\widetilde{m}}
\newcommand{\TM}{\overline{\widetilde{m}}}

\newcommand{\KX}{\overline{X}}
\newcommand{\Sym}{\mathrm{Sym}}
\newcommand{\groth}{\overline{s}}
\newcommand{\dgroth}{\underline{s}}
\newcommand{\cc}{\mathsf{c}}
\newcommand{\SSYT}{\mathrm{SSYT}}
\newcommand{\SV}{\mathrm{SV}}
\newcommand{\trans}{\mathsf{T}}
\newcommand{\SSC}{\mathsf{SSC}}
\newcommand{\Gr}{\mathrm{Gr}}
\newcommand{\id}{\mathrm{id}}

\newcommand{\threeone}{{\bf 3} + {\bf 1}}

\newcommand{\mch}[2]{
\left.\mathchoice
  {\left(\kern-0.48em\binom{#1}{#2}\kern-0.48em\right)}
  {\big(\kern-0.30em\binom{\smash{#1}}{\smash{#2}}\kern-0.30em\big)}
  {\left(\kern-0.30em\binom{\smash{#1}}{\smash{#2}}\kern-0.30em\right)}
  {\left(\kern-0.30em\binom{\smash{#1}}{\smash{#2}}\kern-0.30em\right)}
\right.}

\DeclareMathOperator{\sgn}{sgn}
\DeclareMathOperator{\SSS}{\mathsf{SS}}

\DeclareMathOperator{\ts}{\mathsf{ts}}

\newcommand{\ebox}{\overline{e}}
\newcommand{\estar}{\overline{e}'}

\newtheorem{lemma}{Lemma}[section]

\newtheorem{corollary}[lemma]{Corollary}
\newtheorem{theorem}[lemma]{Theorem}

\newtheorem{conjecture}[lemma]{Conjecture}
\newtheorem{problem}[lemma]{Problem}
\newtheorem{proposition}[lemma]{Proposition}

\theoremstyle{remark}
\newtheorem{definition}[lemma]{Definition}

\newtheorem{remark}[lemma]{Remark}
\newtheorem{examplex}[lemma]{Example}
\newenvironment{example}
  {\pushQED{\qed}\examplex}
  {\popQED\endexamplex}
  
\numberwithin{equation}{section}

\usepackage{todonotes}

\title{The Kromatic Symmetric Function: A $K$-theoretic analogue of $X_G$}
\author{Logan Crew}
\author{Oliver Pechenik}
\author{Sophie Spirkl}
\address{Department of Combinatorics \& Optimization, University of Waterloo, Waterloo, ON, N2L 3G1, Canada.} 
\email{{\tt \{lcrew, opecheni, sspirkl\}@uwaterloo.ca}}

\date{\today}

\keywords{chromatic symmetric function, Grothendieck polynomial, $K$-theory, deletion--contraction relation, Stanley--Stembridge conjecture}
\subjclass{05C15, 05C31, 05E05}

\begin{document}

\begin{abstract}
Schur functions are a basis of the symmetric function ring that represent Schubert cohomology classes for Grassmannians. Replacing the cohomology ring with $K$-theory yields a rich combinatorial theory of inhomogeneous deformations, where Schur functions are replaced by their $K$-analogues, the basis of \emph{symmetric Grothendieck functions}.
We introduce and initiate a theory of the \emph{Kromatic symmetric function} $\KX_G$, a $K$-theoretic analogue of the chromatic symmetric function $X_G$ of a graph $G$. The Kromatic symmetric function is a generating series for graph colorings in which vertices may receive any nonempty set of distinct colors such that neighboring color sets are disjoint. 

Our main result lifts a theorem of Gasharov (1996) to this setting, showing that when $G$ is a claw-free incomparability graph, $\KX_G$ is a positive sum of symmetric Grothendieck functions. This result suggests a topological interpretation of Gasharov's theorem. We then show that the Kromatic symmetric functions of path graphs are not positive in any of several $K$-analogues of the $e$-basis of symmetric functions, demonstrating that the Stanley--Stembridge conjecture (1993) does not have such a lift to $K$-theory and so is unlikely to be amenable to a topological perspective.
We also define a vertex-weighted extension of $\KX_G$ and show that it admits a deletion--contraction relation. Finally, we give a $K$-analogue for $\KX_G$ of the classic monomial-basis expansion of $X_G$.
\end{abstract}

\maketitle

\section{Introduction}

The \emph{chromatic symmetric function} $X_G$ of a graph $G$ was introduced by R.~Stanley \cite{stanley} as a generalization of G.D.~Birkhoff's \emph{chromatic polynomial} \cite{birkhoff}. While the chromatic polynomial enumerates proper graph colorings by the number of colors used, $X_G$ also records how many times each color is used. A recent boom of research regarding $X_G$ has focused on the \emph{Stanley--Stembridge conjecture} \cite{stanley.stembridge}, which proposes (in a reformulation by M.~Guay-Paquet \cite{guay}) that unit interval graphs have chromatic symmetric functions that expand positively in the $e$-basis of the ring $\Sym$ of symmetric functions. In the last few years, various special cases of this conjecture have been established through direct combinatorial analysis, including the cases of \emph{lollipop graphs} \cite{dahl} and many \emph{claw-free} graphs \cite{hamel2019}.
Another approach has been to consider various generalizations of the chromatic symmetric function and corresponding lifts of the Stanley--Stembridge conjecture. Examples of this latter approach include the \emph{chromatic quasisymmetric function} and \emph{Shareshian--Wachs conjecture} of \cite{shareshian.wachs} (further studied in \cite{abreu, per2022, cho2022, colm}), the \emph{chromatic nonsymmetric functions} of J.~Haglund--A.~Wilson \cite{Haglund.Wilson} (further studied in \cite{Tewari.Wilson.Zhang}), and D.~Gebhard--B.~Sagan's \cite{noncomm} chromatic symmetric function in noncommuting variables combined with notions of \emph{$(e)$-positivity} and \emph{appendable $(e)$-positivity} (further studied in \cite{centered, dahlberg2019, dahl2}). Our work provides a novel generalization of $X_G$ in the same vein.

An important appearance of the ring of symmetric functions $\Sym$ is as the cohomology of complex \emph{Grassmannians} (parameter spaces for linear subspaces of a vector space) or more precisely for the classifying space $BU$. Here, the \emph{Schubert classes} derived from a natural cell decomposition of $BU$ are represented by the \emph{Schur function} basis $s_\lambda$ of $\Sym$. A richer perspective into the topology of $BU$ is obtained by replacing cohomology with a generalized cohomology theory. In particular, there has been much focus on studying the associated combinatorics of the \emph{$K$-theory} ring (see \cite{Buch,Monical.Pechenik.Searles,Pechenik.Yong,Thomas.Yong}). In this context, many of the classical objects of symmetric function theory are seen to have interesting $K$-analogues, often resembling ``superpositions'' of classical objects. For example, classical \emph{semistandard Young tableaux} are replaced by \emph{set-valued tableaux} (allowing multiple labels per cell), while Schur functions are replaced by \emph{Grothendieck polynomials} $\groth_\lambda$ (inhomogeneous deformations of $s_\lambda$).

Our work introduces a $K$-analogue of the chromatic symmetric function $X_G$, enumerating colorings of the graph $G$ that assign a nonempty set of distinct colors to each vertex such that adjacent vertices receive disjoint sets. While our \emph{Kromatic symmetric function} $\KX_G$ is new, similar functions have been previously considered. The first such function was originally discussed by R.~Stanley \cite{stanley2} in the context of graph analogues of symmetric functions, with connections to the real-rootedness of polynomials. Recently, as part of his effort to refine Schur-positivity results and the Stanley--Stembridge conjecture, B.-H.~Hwang \cite{hwang} studied a similar quasisymmetric function for graphs endowed with a fixed map $\alpha: V(G) \rightarrow \mathbb{N}$ that dictates the size of the set of colors each vertex receives. To connect chromatic quasisymmetric functions of vertex-weighted graphs to \emph{horizontal-strip LLT polynomials}, F.~Tom \cite{privtom} has considered a variant for fixed $\alpha$ with repeated colors allowed. Our work appears to be the first to connect these ideas to the combinatorics of $K$-theoretic Schubert calculus. (However, \cite{Nenashev.Shapiro} (see also, \cite{Shapiro.Smirnov.Vaintrob}) is similar in spirit to our work, developing a $K$-theoretic analogue of the \emph{Postnikov--Shapiro algebra} \cite{Postnikov.Shapiro}, an apparently unrelated invariant of graphs).

In this paper, having introduced the \emph{Kromatic symmetric function}, we begin to develop its combinatorial theory. We show that the Kromatic symmetric function $\KX_G$ for any graph $G$ expands positively in a $K$-theoretic analogue (that we also introduce) of the \emph{monomial basis} of $\Sym$. In this expansion, the coefficients enumerate coverings of the graph by (possibly overlapping) \emph{stable sets}. We further extend the definition of $\KX_G$ to a vertex-weighted setting, where we give a deletion--contraction relation analogous to that developed by the first and last authors \cite{delcon} for the vertex-weighted version of $X_G$. 

Our main result is that the Kromatic symmetric function of a \emph{claw-free incomparability graph} expands positively in the symmetric Grothendieck basis $\groth_\lambda$ of $\Sym$, lifting to $K$-theory a celebrated result of V.~Gasharov \cite{gash} that such graphs have Schur-positive chromatic symmetric functions. While all known proofs of Gasharov's theorem are representation-theoretic or purely combinatorial, the existence of our $K$-theoretic analogue suggests that both results likely also have an interpretation in terms of the topology of Grassmannians. Precisely, for each claw-free incomparability graph $G$, there should be a subvariety of the Grassmannian whose cohomology class is represented by $X_G$ and whose $K$-theoretic structure sheaf class is represented by $\KX_G$. It would be very interesting to have an explicit construction of such subvarieties.

On the other hand, we show that the Kromatic symmetric functions $\KX_{P_n}$ of path graphs $P_n$ generally do not expand positively in either of two $K$-theoretic deformations we propose for the $e$-basis of $\Sym$. This fact suggests that the Stanley--Stembridge conjecture, if true, is not naturally interpreted in terms of the cohomology of Grassmannians and is unlikely to be amenable to such topological tools from Schubert calculus. We hope these observations can play a similar role to \cite{Dahlberg.Foley.vanWilligenburg} in limiting the range of potential avenues of attack on the Stanley--Stembridge conjecture.

{\bf This paper is organized as follows.} In Section~\ref{sec:background}, we provide an overview of the background and notation used from symmetric function theory (Section~\ref{sec:sym}), $K$-theoretic Schubert calculus (Section~\ref{sec:K}), and graph theory (Section~\ref{sec:graph}). In Section~\ref{sec:Kromatic}, we formally introduce the Kromatic symmetric function $\KX_G$ and give its basic properties, including a formula for the expansion in a new $K$-analogue of the monomial basis of $\Sym$ and a deletion--contraction relation for a vertex-weighted generalization. We also give our main theorem that the Kromatic symmetric functions of claw-free incomparability graphs expand positively in symmetric Grothendieck functions, lifting the main result of \cite{gash}. In Section~\ref{sec:StanleyStembridge}, we introduce two different $K$-theoretic analogues of the $e$-basis of $\Sym$ and show that the Kromatic symmetric function $\KX_{P_3}$ of a $3$-vertex path graph $P_3$ is not positive in either analogue, casting doubt on hopes for a Schubert calculus-based approach to the Stanley--Stembridge conjecture.

\section{Background}\label{sec:background}

Throughout this work, $\mathbb{N}$ denotes the set of (strictly) positive integers. We write $[n]$ for the set of positive integers $\{1, 2, \dots, n\}$. If $S$ is any set, $2^S$ denotes the power set of all subsets of $S$.

\subsection{Partitions and symmetric functions}\label{sec:sym}

In this section, we give a brief overview of necessary background material necessary. Further details can be found in the textbooks of Stanley \cite{stanleybook}, Manivel \cite{manivel}, and Macdonald \cite{mac}.

An \newword{integer partition} $\lambda = (\lm_1 \geq \lm_2 \geq  \dots \geq \lm_k)$ is a finite nonincreasing sequence of positive integers. We define $\ell(\lm)$ to be the length of the sequence $\lm$ (so above, $\ell(\lm) = k$). We define $r_i(\lm)$ to be the number of occurrences of $i$ as a part of $\lm$ (so, for example, $r_1(2,1,1,1) = 3$). If 
\[ 
\sum_{i=1}^{\ell(\lm)} \lm_i = n,
\]
we say that $\lm$ is a partition of $n$, and we write $\lm \vdash n$.
The \newword{Young diagram of shape} $\lm$ is a set of squares called \newword{cells}, left- and top-justified (that is, in ``English notation''), such that the $i$th row from the top contains $\lm_i$ cells. For example, the Young diagram of shape $(2,2,1)$ is $\ytableausetup{smalltableaux}
\ydiagram{2,2,1}$. Let $C(\lm)$ denote the set of cells of the Young diagram of shape $\lm$. If $\cc \in C(\lm)$ is a cell of the Young diagram of shape $\lambda$, we write $\cc^\uparrow$ for the cell immediately above $\cc$ (assuming it exists), $\cc^\rightarrow$ for the cell immediately right of $\cc$, and so on. We write $\lambda^\trans$ for the \newword{transpose} of $\lambda$, the integer partition whose Young diagram is obtained from that of $\lambda$ by exchanging rows and columns.

Let $S_\mathbb{N}$ denote the set of all permutations of the set $\mathbb{N}$ fixing all but finitely-many elements.
A \newword{symmetric function} $f \in \mathbb{C} \llbracket x_1,x_2,\dots, \rrbracket$ is a power series of bounded degree such that for each permutation $\sigma \in S_{\mathbb{N}}$, we have $f(x_1,x_2,\dots) = f(x_{\sigma(1)}, x_{\sigma(2)}, \dots)$. The set $\Sym \subset \mathbb{C}\llbracket x_1,x_2,\dots \rrbracket$ of symmetric functions forms a $\mathbb{C}$-vector space. Furthermore, if $\Lm^d$ denotes the set of symmetric functions that are homogeneous of degree $d$, then each $\Sym^d$ is a vector space, and \[
\Sym = \bigoplus_{d=0}^{\infty} \Sym^d
\]
as graded vector spaces. 

The dimension of $\Sym^d$ as a $\mathbb{C}$-vector space is equal to the number of integer partitions of $d$, and many bases of symmetric functions are conveniently indexed by integer partitions. Below we provide some commonly used bases that will be used in this paper.

\begin{definition}
The following are bases of $\Sym$:
\begin{itemize}
    \item the \newword{monomial symmetric functions} $\{m_{\lm}\}$, defined as
    \[
    m_{\lm} = \sum x_{i_1}^{\lm_1} \dots x_{i_{\ell(\lm)}}^{\lm_{\ell(\lm)}},
    \]
    where the sum ranges over all distinct monomials formed by choosing distinct positive integers $i_1, \dots, i_{\ell(\lm)}$;
    \item the \newword{augmented monomial symmetric functions} $\{ \tm_{\lm} \}$, defined as
    \[
    \tm_{\lm} = \left(\prod_{i=1}^{\infty} r_i(\lm)! \right) m_{\lm};
    \]
    \item the \newword{elementary symmetric functions} $\{e_\lambda \}$, defined by 
    \[
    e_n = \prod_{i_1 < \dots < i_n} x_{i_1} \dots x_{i_n}; \quad e_{\lm} = e_{\lm_1} \dots e_{\lm_{\ell(\lm)}}; 
    \]
    \item and the \newword{complete homogeneous symmetric functions} $\{h_\lambda \}$, defined by
    \[
    h_n = \prod_{i_1 \leq \dots \leq i_n} x_{i_1} \dots x_{i_n}; \quad h_{\lm} = h_{\lm_1} \dots h_{\lm_{\ell(\lm)}}.
    \]
\end{itemize}
\end{definition}

The space of symmetric functions is equipped with a natural inner product $\langle \cdot, \cdot \rangle$; it may be defined by
\[
\langle h_{\lm}, m_{\mu} \rangle = \delta_{\lm,\mu},
\]
where $\delta_{\bullet, \bullet}$ denotes the \emph{Kronecker delta function}. 

We will also need the basis of Schur functions.  A \newword{Young tableau} of shape $\lm$ is a function $T: C(\lm) \rightarrow \mathbb{N},$
typically visualized by writing the value $T(\cc)$ in the cell $\cc$. A Young tableau $T$ of shape $\lm$ is \newword{semistandard} if for each cell $\cc \in C(\lm)$, we have $T(\cc) \leq T(\cc^\rightarrow)$ and $T(\cc) < T(\cc^\downarrow)$ whenever the cells in question exist. We write $\SSYT(\lambda)$ for the set of all semistandard Young tableaux of shape $\lambda$.
The \newword{Schur function} $s_{\lm}$ is defined by
\[
s_{\lm} = \sum_{T\in \SSYT(\lambda)} x^T, \quad \text{where} \quad
x^T = \prod_{\cc \in C(\lm)} x_{T(\cc)}.\] As $\lambda$ ranges over integer partitions, the Schur functions are another basis of $\Sym$. The  inner product on $\Sym$ also satisfies
\[
\langle s_{\lm}, s_{\mu} \rangle = \delta_{\lm,\mu}.
\]

When $f \in \Sym$ is a symmetric function and $\{b_{\lm}\}$ is a basis of symmetric functions indexed by integer partitions $\lm$, the notation $[b_{\mu}]f$ denotes the coefficient of $b_{\mu}$ when $f$ is expanded in the $b$-basis. A symmetric function $f \in \Sym$ is said to be \newword{$b$-positive} if $[b_{\mu}]f$ is nonnegative for every integer partition $\mu$.

\subsection{$K$-theoretic Schubert calculus}\label{sec:K}
The \newword{Grassmannian} $\Gamma_k = \Gr_k(\mathbb{C}^\infty)$ is the parameter space of $k$-dimensional vector subspaces of the space of all eventually-zero sequences of complex numbers. 
The space $\Gamma_k$ can be given the structure of a projective Ind-variety and has a cell decomposition into cells $\Gamma_\lambda$ indexed by partitions with at most $k$ parts. Each $\Gamma_\lambda$ induces a cohomology class $\sigma_\lambda \in H^\star(\Gamma_k)$ and classically we have $H^\star(\Gamma_k)  \cong \Sym_k = \Sym \cap \mathbb{C}[x_1, \dots, x_k]$ with the isomorphism taking the class of the cell $\sigma_\lambda$ to the Schur polynomial $s_\lambda(x_1, \dots, x_k)$.

Each cell-closure in $\Gamma_k$ also has a structure sheaf, inducing a class in the representable \newword{$K$-theory} ring $K^0(\Gamma_k)$. These $K$-theoretic classes are represented by inhomogeneous symmetric polynomials called \emph{Grothendieck polynomials} $\groth_\lambda(x_1, \dots, x_k)$.

A \newword{set-valued tableau} of shape $\lm$ is a filling $T$ of each cell of $C(\lm)$ with a nonempty \emph{set} of positive integers. The set-valued tableau $T$ is \newword{semistandard} if for each cell $\cc \in C(\lm)$, we have $\max T(\cc) \leq \min T(\cc^\rightarrow)$ and $\max T(\cc) < \min T(\cc^\downarrow)$ whenever the cells in question exist.
In other words, $T$ is semistandard if every Young tableau formed by choosing one number from the set of each cell is semistandard. Let $\SV(\lambda)$ denote the set of all semistandard set-valued tableaux of shape $\lambda$.
The \newword{symmetric Grothendieck function} $\groth_{\lm}$ is
\[
\groth_{\lm} = \sum_{T \in \SV(\lm)} (-1)^{|T|-\ell(\lm)}x^T,
\]
where $|T| = \sum_{\cc \in C(\lm)} |T(\cc)|$ and $x^T = \prod_{\cc \in C(\lm)} \prod_{i \in T(\cc)} x_i$.
Note that $\groth_{\lm}$ contains terms of degree greater than or equal to $|\lm|$, and that the sum of all of its lowest-degree terms is equal to $s_{\lm}$. This tableau formula for $\groth_\lambda$ is due to A.~Buch \cite{Buch}. For further background on $K$-theoretic Schubert calculus and symmetric Grothendieck functions, see \cite{Monical.Pechenik.Searles,Pechenik.Yong}.

We will also need the \newword{dual symmetric Grothendieck function} $\dgroth_\lambda$ defined by 
\[
\langle \groth_{\lm}, \dgroth_{\mu} \rangle = \delta_{\lm,\mu}.
\] Dual symmetric Grothendieck functions were first introduced explicitly in \cite{Lam.Pylyavskyy} in relation to the \emph{$K$-homology} of $\Gamma_k$; however, they are also implicit in the earlier work \cite{Buch}. Each $\dgroth_\lambda$ contains terms of degree less than or equal to $|\lm|$; moreover, the sum of all of its lowest-degree terms is equal to $s_{\lm}$. Although an attractive tableau formula for $\dgroth_\lambda$ was given in \cite{Lam.Pylyavskyy}, we do not recall it here, as we will not need it.

\subsection{Graphs and coloring}\label{sec:graph}

Here, we recall basic notions, terminology, and notations from graph theory. For further details, see the textbooks \cite{Diestel,West}.

A \newword{graph} $G$ consists of a set $V$ of \newword{vertices}, and a set $E$ of unordered pairs of distinct vertices called \newword{edges}. All graphs in this paper are \emph{simple}, so there are no loops, and no multi-edges. When $\{v_1,v_2\} \in E(G)$, we will typically denote this edge by $v_1v_2$ and say $v_1$ and $v_2$ are \newword{adjacent}. Two graphs $G,G'$ are \newword{isomorphic} if there is a bijection $\phi : V(G) \to V(G')$ such that, for all vertices $v, w \in V(G)$, we have $vw \in E(G)$ if and only if $\phi(v) \phi(w) \in E(G')$. In this paper, we consider graphs up to isomorphism.

The \newword{complete graph} $K_d$ with $d$ vertices is the graph such that $V(K_d) = [d]$, and \[E(K_d) = \{vw: v, w \in [d], v \neq w\}.\] The $n$-vertex \newword{path} $P_n$ has vertex set $V(P_n) = [n]$ and edge set $E(P_n) = \{ uv : u,v \in [n], v-u = 1\}$. The \newword{claw} $K_{1,3}$ has vertex set $V(K_{1,3}) = [4]$ and edge set $E(K_{1,3}) = \{\{1,2\},\{1,3\},\{1,4\}\}$.

An \newword{induced subgraph} of a graph $G$ is a graph $H$ such that $V(H) \subseteq V(G)$ and 
\[
E(H) = \{ vw \in E(G) : v,w \in V(H)\}.
\] We say the graph $G$ is \newword{$H$-free} if no induced subgraph of $G$ is isomorphic to $H$. We will be especially interested in claw-free graphs.

A \newword{stable set} (or \newword{independent set}) of a graph $G$ is a set $S \subseteq V(G)$ of vertices such that for each $v, w \in S$, $vw \notin E(G)$. A \newword{clique} of a graph $G$ is a set $S \subseteq V(G)$ of vertices such that for each $v \neq w \in S$, $vw \in E(G)$. 

For $\alpha : V(G) \to \mathbb{N}$ a vertex weight function of the graph $G$, the \newword{$\alpha$-clan graph of $G$} is the graph $C_\alpha(G)$ obtained by blowing up each vertex $v$ into a clique of $\alpha(v)$ vertices. More formally, $C_\alpha(G)$ has vertex set 
$
V(C_\alpha(G)) = \{(v,i) : v \in V(G), i \in [\alpha(v)] \}.
$
In $C_\alpha(G)$, the vertices $(v,i)$ and $(w,j)$ are adjacent either if $vw \in E(G)$ or if both $v=w$ and $i \neq j$. 

Given a vertex $v \in V(G)$, its \newword{open neighborhood} $N(v)$ is defined by $N(v) = \{w: vw \in E(G)\}$. Given $S \subseteq V(G)$ and $v \in V(G)$ with $v \notin S$, we let $vS \subseteq E(G)$ denote the set of edges $\{vs: s \in S\}$.
The \newword{contraction} of a graph $G$ by a pair of distinct vertices $v, w \in V(G)$, denoted $G/vw$, is the graph with vertex set 
\[
V(G/vw) = \left(V(G) \bk \{v,w\}\right) \cup \{z_{vw}\},
\]
where $z_{vw}$ is a new vertex,
and edge set 
\[
E(G/vw) = \left(E(G) \bk \big( vN(v) \cup wN(w) \big)\right) \cup \big( z_{vw} N(v) \cup z_{vw} N(w) \big).
\]

A \newword{coloring} of a graph $G$ is a function $\kappa: V(G) \rightarrow \mathbb{N}$. A coloring $\kappa$ of $G$ is \newword{proper} if $\kappa(a) \neq \kappa(b)$ whenever $ab \in E(G)$. 

The \newword{chromatic symmetric function} \cite{stanley}  of a graph $G$ is the power series
\[
X_{G} = \sum_{\kappa} \prod_{v \in V(G)} x_{\kappa(v)}
\]
where the first sum ranges over all proper colorings $\kappa$ of $G$. Note that, for every graph $G$, $X_G \in \Sym$.

\subsection{Posets and their incomparability graphs}
A \newword{poset} (partially-ordered set) $(P, \leq)$ is a set $P$ together with a binary relation $\leq$ that is \newword{transitive} ($a \leq b$ and $b \leq c$ implies $a \leq c$), \newword{reflexive} ($a \leq a$), and \newword{weakly antisymmetric} ($a\leq b$ and $b \leq a$ implies $a = b$). For $a,b \in P$, we write $a < b$ if $a \leq b$ and $a \neq b$. We often write $P$ as shorthand for $(P, \leq)$ and decorate the relation as $\leq_P$ for clarity as needed. For more background on posets than is provided here, see \cite{West}. 

When $a,b \in P$ are such that $a \not \leq b$ and $b \not \leq a$, we say $a$ and $b$ are \newword{incomparable}. We write $\bf n$ for the unique totally ordered $n$-element poset and call such a poset a \newword{chain}. The \newword{sum} $P + Q$ of posets $(P, \leq_P), (Q, \leq_Q)$ is the disjoint union of sets $P \sqcup Q$ with the relation $a \leq_{P+Q} b$ if and only if either $a,b \in P$ with $a \leq_P b$ or $a,b \in Q$ with $a \leq_Q b$.

We say $(Q, \leq_Q)$ is a \newword{subposet} of $(P, \leq_P)$ if $Q$ is a subset of $P$ and, for all $a,b \in Q$, we have $a \leq_Q b$ if and only if $a \leq_P b$. Two posets $(P, \leq_P), (Q, \leq_Q)$ are \newword{isomorphic} if there is a bijection $\phi : P \to Q$ such that, for all $a,b \in P$, we have $a \leq_P b$ if and only if $\phi(a) \leq_Q \phi(b)$. If $(P, \leq_P), (Q, \leq_Q)$ are any two posets, we say that $(P, \leq_P)$ is \newword{$(Q, \leq_Q)$-free} if no subposet of $P$ is isomorphic to $Q$. We will be mostly interested in posets that are $(\threeone)$-free.

Associated to any poset $P$ is its \newword{incomparability graph} $I(P)$. This is the graph whose vertex set is $V(I(P)) = P$ and whose edge set is $E(I(P)) = \{ab: a,b \in P, a \not \leq b, b \not \leq a \}$. That is to say, edges connect incomparable elements of the poset. It is straightforward to see that the poset $P$ is $(\threeone)$-free if and only if its incomparability graph is claw-free; however, many claw-free graphs are not incomparability graphs of posets.

\section{The Kromatic Symmetric Function}\label{sec:Kromatic}
\subsection{Main definition}\label{sec:main_def}

A \newword{vertex-weighted graph} $(G,\alpha)$ consists of a graph $G$ together with a function 
$\alpha \colon V(G) \rightarrow \mathbb{N};
$
we call $\alpha$ the \newword{weight function} on the vertices of $G$.
A \newword{proper $\alpha$-coloring} of $G$ is a function 
$
\kappa : V(G) \to 2^\mathbb{N} \bk \{\emptyset \}
$
assigning to each $v \in V(G)$ a set of $\alpha(v)$ distinct colors in $\mathbb{N}$, subject to the constraint that when $uv \in E(G)$, we have $\kappa(u) \cap \kappa(v) = \emptyset$. Note that these conditions are equivalent to saying that every choice of a single element from each $\kappa(v)$ yields a proper coloring of $G$. A \newword{proper set coloring} of $G$ is a proper $\alpha$-coloring for some weight function on the vertices of $G$.

The \newword{set chromatic symmetric function} of the vertex-weighted graph $(G,\alpha)$ is
\[
X_{G}^\alpha = \sum_{\kappa} \prod_{v \in V(G)} \prod_{i \in \kappa(v)} x_i,
\]
where the first sum runs over all proper $\alpha$-colorings of $G$. Note that up to a scalar factor depending only on $\alpha$, the set chromatic symmetric function $X_{G}^\alpha$ equals the chromatic symmetric function $X_{C_\alpha(G)}$ of the $\alpha$-clan graph of $G$. 

\begin{definition}\label{def:Kromatic}
The \newword{Kromatic symmetric function} of a graph $G$ is the symmetric power series
\[
\KX_G = \sum_{\alpha} X_{G}^\alpha,
\]
where $\alpha$ ranges over all weight functions of the vertex set $V(G)$.
\end{definition}

In other words, $\KX_G$ enumerates all colorings of $G$ by nonempty sets of colors, such that adjacent vertices receive disjoint sets of colors. Note that $\KX_G$ is not a homogeneous symmetric function, but rather consists of $X_G$ plus terms of degree higher than $|V(G)|$. 

\begin{remark}
Stanley~\cite{stanley2} considered a function $Y_G$ related to $\KX_G$, although with two differences. Firstly, $Y_G$ uses the rescaled power series $X_{C_\alpha(G)}$ in place of $X_{G}^\alpha$. Secondly, $Y_G$ allows $\alpha(v) = 0$, whereas the Kromatic symmetric function $\KX_G$ only considers strictly positive vertex weights. We are unaware of any further study of the functions $Y_G$ since their introduction in \cite{stanley2}.
\end{remark}

\begin{remark}
    It is easy to observe that the Kromatic symmetric function $\KX_G$ of any graph $G$ is $m$-positive. Moreover, one may also check that $\KX_G$ is positive in the basis $\{\omega(p_\lambda)\}_\lambda$, where $p_\lambda$ denotes the \emph{power sum symmetric function} and $\omega$ is the standard involution on symmetric functions.
\end{remark}

Although weight functions are used in the definition of $\KX_G$, the function $\KX_G$ is independent of any particular one.
We will find it useful to also consider a vertex-weighted analogue of $\KX_G$. Let $\alpha$ and $\omega$ be independent vertex weight functions on $G$. Define
\[
X_{(G,\omega)}^\alpha = \sum_{\kappa} \prod_{v \in V(G)} \left(\prod_{i \in \kappa(v)} x_i\right)^{\omega(v)},
\]
where again the first sum runs over all proper $\alpha$-colorings of $G$.
Finally, we define the \newword{vertex-weighted Kromatic symmetric function} of the vertex-weighted graph $(G,\omega)$ to be
\[
\KX_{(G,\omega)} = \sum_{\alpha} X_{(G,\omega)}^\alpha,
\]
where the sum is over all weight functions $\alpha$. In this way, $\KX_{(G,\omega)}$ is a generating function for proper set colorings of $G$.

\subsection{A $K$-theoretic monomial expansion}\label{sec:Kmonomial}
For $\lambda$ an integer partition, let $K_\lambda$ denote the vertex-weighted complete graph $(K_{\ell(\lambda)}, \omega)$, where $\omega(i) = \lambda_i$ for each $i$.
It is straightforward to see that $X_{K_\lambda} = \tm_{\lm}$, the augmented monomial symmetric function. Thus, by analogy, we define
\[
\TM_{\lm} := \KX_{K_{\lm}} = \sum_{\alpha} \tm_{\lm_1^{\alpha_1}, \dots, \lm_{\ell(\lm)}^{\alpha(\ell(\lm))}},
\]
where the sum is over all vertex weight functions $\alpha$ of $K_{\ell(\lambda)}$. We call $\TM_{\lm}$ the \newword{$K$-theoretic augmented monomial symmetric function}. To justify this definition, we show that the Kromatic symmetric function of every graph (even every vertex-weighted graph) is a positive sum of $K$-theoretic augmented monomial symmetric functions. 

First, we need some additional definitions.
We define a \newword{stable set cover} $C$ of  a graph $G$ to be a collection of (distinct) stable sets of $G$ such that every vertex of $V(G)$ is in at least one element of $C$. In symbols, this means that
\[
\bigcup_{S \in C} S = V(G);
\]
note that this union is not required to be disjoint. We write $\SSC(G)$ for the family of all stable set covers of $G$.
For $C \in \SSC(G)$, if $G$ is endowed with a vertex weight function $\omega$, let $\lambda(C)$ be the partition of length $|C|$ whose parts are $\sum_{v \in S} \omega(v)$ for $S \in C$.
Finally, let the \newword{color class} of the color $i$ in a proper set coloring $\kappa$ be 
\[
\{v \in V(G) : i \in \kappa(v)\},
\]
the set of vertices of $G$ that receive color $i$ (possibly among other colors) under $\kappa$.

\begin{proposition}\label{lem:mbasis}
For any vertex-weighted graph $(G,\omega)$, we have
\[
\KX_{(G,\omega)} = \sum_{C \in \SSC(G)} \TM_{\lm(C)}.
\]
\end{proposition}
 \begin{proof}
 The monomials of $\KX_{(G,\omega)}$ correspond to proper set colorings $\kappa$ of $G$. For each such $\kappa$, note that the set of its color classes is a stable set cover $C_\kappa$ of $G$. 
 
 For each $C\in \SSC(G)$, the monomials of $\TM_{\lm(C)}$ enumerate all proper set colorings $\kappa$ of $G$ such that
 \begin{itemize}
     \item each $S \in C$ is the color class of at least one color $i$ under $\kappa$; and
     \item for each nonempty $T \subseteq V(G)$ with $T \notin C$, there is no color $j$ such that $T$ is the color class of $j$ under $\kappa$.
 \end{itemize}
 In other words, the monomials of $\TM_{\lm(C)}$ correspond to all proper set colorings $\kappa$ of $G$ such that $C_\kappa =C$.
 
Since this correspondence between the monomials of $\KX_{(G,\omega)}$ and those of $\sum_{C \in \SSC(G)} \TM_{\lm(C)}$ is a weight-preserving bijection, the two power series are equal.
 \end{proof}

 For some small graphs $G$, the expansions of $\KX_G$ in the $\TM_{\lm}$-basis and in the classical $p$-basis are collected in Table~\ref{tab:mytab}.

 \begin{table}[htb]
\scalebox{1}{
\begin{tabular}{|l|l|l|}
\hline
Graph $G$ \begin{tikzpicture}[scale=1]
  \node at (0, 0)(1){};
    \node at (0,0.2)(void) {};
  \end{tikzpicture} \!\!\!\!\!\!\!\! & $\KX_G$ in the $\TM$-basis & $\KX_G$ in the $p$-basis 
\\ \hline
  \begin{tikzpicture}[scale=1]
  \node[dot] at (0, 0)(1){};
    \node[dot] at (1, 0)(2){};
    \node[dot] at (0.5, 0.866)(3){};
    \node at (0.5,1)(void) {};
    \draw[black, thick] (3) -- (1);
    \draw[black, thick] (2) -- (1);
  \end{tikzpicture} & 
  $\TM_{(1^3)}+\TM_{(2,1)}+2\TM_{(2,1^2)}+\TM_{(2,1^3)}$ & \pbox[c]{5cm}{$p_{(1^3)}-2p_{(2,1)}+p_{(3)}-4p_{(2,1^2)} +p_{(2^2)}+7p_{(3,1)}-4p_{(4)}+ {\rm h.o.t.}$} \\ \hline
    \begin{tikzpicture}[scale=1]
  \node[dot] at (0, 0)(1){};
    \node[dot] at (1, 0)(2){};
    \node[dot] at (0.5, 0.866)(3){};
    \node at (0.5,1)(void) {};
    \draw[black, thick] (3) -- (1);
    \draw[black, thick] (2) -- (1);
    \draw[black, thick] (2) -- (3);
  \end{tikzpicture} & 
  $\TM_{(1^3)}$ & \pbox[c]{5cm}{$p_{(1^3)}-3p_{(2,1)}+2p_{(3)}-6p_{(2,1^2)}+3p_{(2^2)}+12p_{(3,1)}-9p_{(4)}+ {\rm h.o.t.}$} \\ \hline
   \begin{tikzpicture}[scale=1]
  \node[dot] at (0, 0)(1){};
    \node[dot] at (1, 0)(2){};
    \node[dot] at (1, 1)(3){};
     \node[dot] at (0, 1)(4){};
    \draw[black, thick] (3) -- (2);
    \draw[black, thick] (2) -- (1);
    \draw[black, thick] (4) -- (3);
    \draw[black, thick] (1) -- (4); 
  \end{tikzpicture} & \pbox[b]{5cm}{$\TM_{(1^4)}+2\TM_{(2,1^2)}+\TM_{(2^2)}+4\TM_{(2,1^3)}+4\TM_{(2^2,1)}+2\TM_{(2,1^4)}+6\TM_{(2^2,1^2)}+4\TM_{(2^2,1^3)}+\TM_{(2^2,1^4)}$} & \pbox[b]{5cm}{$p_{(1^4)}-4p_{(2,1^2)}+2p_{(2^2)}+4p_{(3,1)}-3p_{(4)}-8p_{(2,1^3)}+12p_{(2^2,1)}+20p_{(3,1^2)}-16p_{(3,2)}-28p_{(4,1)}+20p_{(5)}+ {\rm h.o.t.}$}\\ \hline
\begin{tikzpicture}[scale=1][t]
  \node[dot] at (0, 0)(1){};
    \node[dot] at (1, 0)(2){};
    \node[dot] at (1, 1)(3){};
     \node[dot] at (0, 1)(4){};
    \node at (1,1.1)(void){};
    \draw[black, thick] (3) -- (1);
    \draw[black, thick] (2) -- (1);
    \draw[black, thick] (4) -- (1);
    \draw[black, thick] (4) -- (3);
    \draw[black, thick] (2) -- (4); 
  \end{tikzpicture} & $\TM_{(1^4)}+\TM_{(2,1^2)}+2\TM_{(2,1^3)}+\TM_{(2,1^4)}$ & \pbox[b]{5cm}{$p_{(1^4)}-5p_{(2,1^2)}+2p_{(2^2)} +6p_{(3,1)}-4p_{(4)}-10p_{(2,1^3)}+28p_{(3,1^2)}+16p_{(2^2,1)}-22p_{(3,2)}-42p_{(4,1)}+30p_{(5)} + {\rm h.o.t.}$} \\ \hline
  \begin{tikzpicture}[scale=1]
  \node[dot] at (0, 0)(1){};
    \node[dot] at (1, 0)(2){};
    \node[dot] at (1, 1)(3){};
     \node[dot] at (0, 1)(4){};
    \draw[black, thick] (3) -- (1);
    \draw[black, thick] (2) -- (1);
    \draw[black, thick] (1) -- (4);
  \end{tikzpicture} & \pbox[b]{8.5cm}{$\TM_{(1^4)}+3\TM_{(2,1^2)}+\TM_{(3,1)}+6\TM_{(2,1^3)}+3\TM_{(2^2,1)}+3\TM_{(3,1^2)}+3\TM_{(2,1^4)}+9\TM_{(2^2,1^2)}+3\TM_{(3,1^3)}+3\TM_{(3,2,1)}+9\TM_{(2^2,1^3)}+\TM_{(2^3,1)}+9\TM_{(3,2,1^2)}+\TM_{(3,1^4)}+3\TM_{(2^2,1^4)}+3\TM_{(2^3,1^2)}+9\TM_{(3,2,1^3)}+3\TM_{(3,2^2,1)}+9\TM_{(3,2^2,1^2)}+3\TM_{(3,2,1^4)}+3\TM_{(2^3,1^3)}+9\TM_{(3,2^2,1^3)}+\TM_{(3,2^3,1)}+\TM_{(2^3,1^4)}+3\TM_{(3,2^2,1^4)}+3\TM_{(3,2^3,1^2)}+3\TM_{(3,2^3,1^3)}+\TM_{(3,2^3,1^4)}$} & \pbox[b]{5cm}{$p_{(1^4)}-3p_{(2,1^2)}+3p_{(3,1)}-p_{(4)}-6p_{(2,1^3)}+3p_{(2^2,1)}+15p_{(3,1^2)}-3p_{(3,2)} -16p_{(4,1)} +7p_{(5)} + {\rm h.o.t.}$} \\ \hline
\end{tabular}}
\caption{Kromatic symmetric functions of some small graphs as determined by implementing Proposition~\ref{lem:delcon} in Python, expressed in the $K$-theoretic $\TM$-basis, as well as in the usual $p$-basis. Since the latter expansion is infinite, we write explicitly only the terms of degree at most $|V(G)|+1$, suppressing higher order terms (``h.o.t.'').}
\label{tab:mytab}
\end{table}

\begin{remark}
It is also natural to ask if there is a $K$-theoretic deformation of the $p$-basis that lifts classical $p$-basis expansions of $X_G$ to expansions of $\KX_G$. Since $p_n$ is the chromatic symmetric function of a single vertex of weight $n$, it is natural to attempt to define a $K$-theoretic $p$-basis by letting $\overline{p}_n$ be the Kromatic symmetric function of a single vertex of weight $n$.

However, with this choice it is unclear if the usual $p$-basis expansion of $X_G$ extends to $\overline{p}$-basis expansion of $\KX_G$. Attempting to naively modify Stanley's inclusion--exclusion proof \cite[Theorem 2.5]{stanley} of this expansion for unweighted graphs fails because it uses the fact that if we take a connected graph $G$ and evaluate $\sum_{\kappa} \prod_{v \in V(G)} x_{\kappa(v)}$ over all $\kappa$ such that adjacent vertices receive the same color, this yields $p_{|V(G)|}$, since all vertices must have the same color. But in the Kromatic case, the corresponding statement is that color sets of adjacent vertices have nonempty intersection, which yields many possibilities for the corresponding sum over all such colorings. In particular, the result depends on more than just $|V(G)|$, making analysis more difficult. Indeed, numerical evidence suggests that the $p$-basis expansion of $X_G$ does not directly extend to a $\overline{p}$-basis expansion of $\KX_G$, suggesting that we require either a different $K$-theoretic $p$-basis or a modified expansion.
\end{remark}

\subsection{A deletion--contraction relation}
The Kromatic symmetric function for vertex-weighted graphs also admits a deletion--contraction relation, analogous to that of \cite{delcon} for the chromatic symmetric function, although somewhat more complicated. We first need to set up some additional notation.

Recall that, given $S \subseteq V(G)$ and $v \in V(G)$ with $v \notin S$, $vS$ denotes the set of edges $\{vs: s \in S\}  \subseteq E(G)$. Let $(G,\omega)$ be a vertex-weighed graph, and let $v, w$ be distinct vertices such that $e = vw \notin E(G)$. The graph $G^\star$ has vertex set 
\[
V(G^\star) = V(G) \cup \{ z^\star\},
\]
where $z^\star$ is a new vertex, and edge set
\[
E(G^\star) = E(G) \cup \{vw, vz^\star, wz^\star\} \cup z^\star N(v) \cup z^\star N(w).
\]
If $G$ has a vertex weight function $\omega$, we define an induced vertex weight function $\omega^\star$ on $G^\star$ by 
\[
\omega^\star(u) = 
\begin{cases}
\omega(v)+\omega(w), & \text{if $u = z^\star$;}\\
\omega(u), & \text{if $u \in V(G)$}.
\end{cases}
\]
We also define graphs $G^1, G^2$ with vertex sets 
\[
V(G^i) = V(G)
\]
and edge sets
\[
E(G^1) = E(G) \cup e \cup vN(w) \quad \text{and} \quad E(G^2) = E(G) \cup e \cup wN(v).
\]
When $G$ has a vertex weight function $\omega$, there are induced vertex weight functions $\omega^i$ on $G^i$ given by
\[
\omega^1(u) = 
\begin{cases}
    \omega(v)+\omega(w), & \text{if $u = v$;} \\
    \omega(u), & \text{otherwise};
\end{cases}
\]
and 
\[
\omega^2(u) = 
\begin{cases}
    \omega(v)+\omega(w), & \text{if $u = w$;} \\
    \omega(u), & \text{otherwise}.
\end{cases}
\]
In the contracted graph $G/e$, we give it the weight function $\omega/e$ defined by
\[
(\omega/e)(u) = 
\begin{cases}
    \omega(v)+\omega(w), & \text{if $u = z_{vw}$;} \\
    \omega(u), & \text{otherwise}.
\end{cases}
\]
Finally, let $G \cup e$ be the graph $(V(G),E(G)\cup\{e\}$).

\begin{proposition}\label{lem:delcon}
Let $(G,\omega)$ be a vertex-weighed graph, and let $v$ and $w$ be distinct vertices such that $e = vw \notin E(G)$. Then
\begin{align}\label{eq:delcon}
\KX_{(G,\omega)} = \KX_{(G / e, \omega / e)}+ \KX_{(G \cup e, \omega)}+ \KX_{(G^1, \omega^1)}+\KX_{(G^2, \omega^2)}+\KX_{(G^\star, \omega^\star)}.
\end{align}
\end{proposition}
 \begin{proof}
 The proof is a direct bijection between the proper set colorings contributing to the left and right sides of Equation~\eqref{eq:delcon} as indicated below. In each case it is straightforward to verify the given correspondence is reversible, and that the monomials produced by the corresponding colorings are identical.
 \begin{itemize}
     \item Proper set colorings $\kappa$ of $(G,\omega)$ such that $\kappa(v) = \kappa(w)$ correspond to all proper set colorings $\kappa/e$ of $(G/e,\omega/e)$ by
     \[
     (\kappa /e)(u) = 
     \begin{cases}
         \kappa(v), & \text{if $u = z_{vw}$;} \\
         \kappa(u), & \text{otherwise}.
     \end{cases}
     \]
     \item Proper set colorings $\kappa$ of $(G,\omega)$ such that $\kappa(v) \cap \kappa(w) = \emptyset$ are in exact correspondence with all the proper set colorings of $(G \cup e, \omega)$.
     \item Proper set colorings $\kappa$ of $(G,\omega)$ such that $\kappa(v) \subsetneq \kappa(w)$ correspond to all proper set colorings $\kappa^1$ of $(G^1, \omega^1)$ by 
     \[
     \kappa^1(u) = 
     \begin{cases}
      \kappa(v), & \text{if $u=v$}; \\
      \kappa(w)\bk \kappa(v), & \text{if $u=w$}; \\
      \kappa(u), & \text{otherwise}.
     \end{cases}
     \]
     \item Proper set colorings $\kappa$ of $(G,\omega)$ such that $\kappa(w) \subsetneq \kappa(v)$ correspond to all proper set colorings $\kappa^2$ of $(G^2, \omega^2)$ by 
         \[
     \kappa^2(u) = 
     \begin{cases}
      \kappa(v) \bk \kappa(w), & \text{if $u=v$}; \\
      \kappa(w), & \text{if $u=w$}; \\
      \kappa(u), & \text{otherwise}.
     \end{cases}
     \]
     \item  Proper set colorings $\kappa$ of $(G,\omega)$ that fit into none of the previous categories (that is, those such that each of the sets 
     \[
     \kappa(v) \cap \kappa(w), \kappa(v) \bk \kappa(w), \kappa(w) \bk \kappa(v)
     \]
     are nonempty) correspond to all the proper set colorings $\kappa^\star$ of $(G^\star,\omega^\star)$ by
     \[
     \kappa^\star(u) = 
     \begin{cases}
      \kappa(v) \cap \kappa(w), & \text{if $u=z^\star$}; \\
    \kappa(v)\bk \kappa(w), & \text{if $u=v$}; \\
      \kappa(w)\bk \kappa(v), & \text{if $u=w$}; \\
      \kappa(u), & \text{otherwise}.
     \end{cases}
     \]
 \end{itemize}
 This completes the proof of the deletion--contraction relation.
 \end{proof}

The deletion--contraction relation of Proposition~\ref{lem:delcon} can be used to yield algorithmically the $\TM_\lambda$-expansion of a Kromatic symmetric function $\KX_{(G,\omega)}$ in an alternative fashion to Proposition~\ref{lem:mbasis}.  Define the \newword{total stability} of a graph $G$ to be $\ts(G) = |\SSS(G)|-|V(G)|$, where $\SSS(G)$ denotes the collection of all stable sets of $G$. Since any single vertex of a graph is a stable set, we may view the total stability as the number of non-trivial stable sets. Thus, note that $\ts(G) \geq 0$ and that equality holds if and only if $G$ is a complete graph.

\begin{corollary}\label{lem:recurse}
Recursively applying Proposition~\ref{lem:delcon} to a vertex-weighted graph $(G,\omega)$ (iteratively applying it to an arbitrary nonedge of each non-complete graph formed) terminates in a sum of Kromatic symmetric functions of vertex-weighted complete graphs, yielding the $\TM_{\lm}$ expansion of $\KX_{(G,\omega)}$.
\end{corollary}
 \begin{proof}
 We proceed by induction on the total stability $\ts(G)$. If $\ts(G) = 0$, then $G$ is a complete graph and the result is trivial.
 Otherwise, it is sufficient to show that after applying Proposition~\ref{lem:delcon} to $(G,\omega)$, each of the resulting five graphs
 \[
 G/e, G \cup e, G^1, G^2, G^\star
 \]
 has strictly smaller total stability than $G$ does. We consider each of these five graphs in turn.
 \begin{itemize}
     \item \uline{($G/e$)}: We have $|V(G/e)| = |V(G)|-1$. On the other hand, 
the stable sets of $G/e$ not containing $z_{vw}$ are in obvious bijection with the stable sets of $G$ containing neither $v$ nor $w$. Moreover, there is a bijection between the stable sets of $G/e$ containing $z_{vw}$ and $\{ S \in \SSS(G) : v, w\in S\}$. Together, this gives a bijection between stable sets of $G/e$ and those stable sets of $G$ which contain either both or neither of $v$ and $w$. Since $\{v\}$ and $\{w\}$ are stable sets of $G$, we have $|\SSS(G/e)| \leq |\SSS(G)| - 2$, and so 
\[
\ts(G/e) = |\SSS(G/e)| - |V(G/e)| \leq \big( |\SSS(G)|  - 2 \big) - \big( |V(G)|-1 \big) = \ts(G) -1,
\]
as needed.
    \item \uline{($G \cup e$)}: We have $|V(G \cup e)| = |V(G)|$. Clearly, $\SSS(G \cup e) \subseteq \SSS(G)$. However, this inclusion is strict since $\{v,w\} \in \SSS(G) \bk \SSS(G \cup e)$. Thus, $\ts(G \cup e) < \ts(G)$.
     \item \uline{($G^1$)}: We have $|V(G^1)| = |V(G)|$. Again, it is clear that $\SSS(G^1) \subseteq \SSS(G)$ and the inclusion is strict since $\{v,w\} \in \SSS(G) \bk \SSS(G^1)$.
     \item \uline{($G^2$)}: The analysis is the same as for $G^1$.
     \item \uline{($G^\star$)}: We have $|V(G^\star)| = |V(G)|+1$. Let $X = \{S \in \SSS(G) : v, w \in S \}$ and let $Y = \SSS(G) \bk X$. There is an obvious injection of $\{ S : \SSS(G^\star) : z^\star \notin S \}$ into $Y$, since $vw \in E(G^\star)$. We may also biject $\{ S : \SSS(G^\star) : z^\star \in S \}$ with $X$ by mapping $S \mapsto (S \bk \{z^\star\}) \cup \{v,w\}$. This latter map is well-defined since such an $S$ does not include $v, w$, or any vertex in $N(v)$ or $N(w)$. Combining these bijections yields a bijection of $\SSS(G^\star)$ with $\SSS(G)$, so $|\SSS(G^\star)| = |\SSS(G)|$. We conclude that 
     \[
     \ts(G^\star) = |\SSS(G^\star)| - |V(G^\star)| \leq |\SSS(G)| - \left( |V(G)|+1 \right) < |\SSS(G)| - |V(G)| = \ts(G),
     \]
     as needed.
 \end{itemize}
 Therefore, the corollary follows by induction on $\ts$.
 \end{proof}

\subsection{Grothendieck positivity}\label{sec:Groth_pos}

In 1996, Gasharov \cite{gash} proved that $X_G$ is Schur-positive for $G$ a claw-free incomparability graph of a poset $P$ by showing that $[s_{\lm}]X_G$ enumerates objects he called $\emph{$P$-tableaux}$. We now define a generalization of these objects that is enumerated by $[\groth_{\lm}]\KX_G$. 
Informally, a \emph{Grothendieck $P$-tableau of shape $\lm$} consists of a $P$-tableau (as defined in \cite{gash}) of shape $\mu$ for some $\mu \subseteq \lm$ layered with a semistandard Young tableau of shape $\lm/\mu$ that also satisfies ``flagging'' restrictions on each row. (Similar flagging conditions appear in \cite{lenart} with relation to Schur-basis expansions of Grothendieck symmetric functions; we do not know a direct relation between our Grothendieck $P$-tableaux and \cite{lenart}, nor with the \emph{flagged tableaux} of \cite{Wachs}.)

\begin{definition}
Let $P$ be a poset and $\lm$ an integer partition. A \newword{Grothendieck $P$-tableau of shape $\lm$} is a filling $T$ of the cells of the Young diagram of $\lm$ with elements of $P \sqcup \mathbb{N}$ such that
\begin{itemize}
    \item the cells filled with elements of $P$ form the Young diagram of some partition $\mu \subseteq \lm$ (and so the cells filled with positive integers form the Young diagram of the skew shape $\lm/\mu$);
    \item for each $p \in P$, there exists at least one cell $c$ with $T(c) = p$,
    \item for each cell $\cc$ with $T(\cc) \in P$, \begin{itemize}
        \item  we have  $T(\cc) <_P T(\cc^{\rightarrow})$, if $T(\cc^{\rightarrow}) \in P$, and
        \item we have $T(\cc) \ngtr_P T(\cc^{\downarrow})$, if $T(\cc^{\downarrow}) \in P$;
    \end{itemize} 
    \item for each cell $\cc$ with $T(\cc) \in \mathbb{N}$, 
    \begin{itemize}
        \item we have $T(\cc) \leq T(\cc^{\rightarrow})$, if $T(\cc^{\rightarrow}) \in \mathbb{N}$, 
        \item we have $T(\cc) < T(\cc^{\downarrow})$, if $T(\cc^{\downarrow}) \in \mathbb{N}$, and
        \item we have $T(\cc) \leq i-1$, if $\cc$ is in row $i$ (in particular, the first row contains no positive integers).
    \end{itemize}
\end{itemize}
\end{definition}

\begin{theorem}\label{thm:grot}
If $G$ is a claw-free incomparability graph, then the Kromatic symmetric function $\KX_G$ is
Grothendieck-positive.
Moreover, the coefficient $[\groth_{\lm}]\KX_G$ counts the number of Grothendieck $P$-tableaux of shape $\lm$.
\end{theorem}
\begin{proof}
The basic structure of our proof is as follows. We use a generalized Jacobi--Trudi formula \cite[Equation~(4)]{Lascoux.Naruse} to write a dual symmetric Grothendieck function as a sum of products of complete homogeneous symmetric functions. Then, for any graph $G$, the inner product of this expression with $\KX_G$ yields a formula for the coefficient of $\groth_\lambda$ in the Grothendieck expansion of $\KX_G$ in terms of its monomial expansion. In the case that $G$ is a claw-free incomparability graph, we then extend Gasharov's \cite{gash} theory of \emph{$P$-arrays} to collect terms in this expansion and extend his sign-reversing involution to cancel all terms except those corresponding to Grothendieck $P$-tableaux.

Now, we give the details of this argument.
Let $G$ be a claw-free incomparability graph and let $P$ be a poset such that $G = I(P)$. The graph $G$ being claw-free is equivalent to the poset $P$ being $(\threeone)$-free.

Fix a positive integer $n \in \mathbb{N}$ and a partition $\lm \vdash n$. Let $k = \ell(\lm)$. Let $N \geq 2n$ be fixed, and let $S_N$ denote the symmetric group of permutations of $[N]$, with identity element $\id_{S_N}$. For $\pi \in S_N$, the \newword{sign} of $\pi$ is 
\[
\sgn \pi = \begin{cases}
    +1, & \text{if $\pi$ is in the \emph{alternating group} $A_N \subset S_N$;} \\
    -1, & \text{otherwise.}
\end{cases}
\]
We will write 
\[
\mch{n}{k} \coloneqq \binom{n+k-1}{k}
\]
as a shorthand for the number of $k$-element \emph{multisets} with elements of $n$ types.

We will use the notation that for $f$ a symmetric function, $f[x_N] \coloneqq f(x_1,x_2,\dots,x_N,0,0,\dots)$ is a restriction to finitely many variables, and $f[x_N+r] \coloneqq f(x_1,\dots,x_N,1,1,\dots,1,0,0,\dots)$ is a restriction to $N+r$ variables with $r$ of them set equal to $1$. Recall the dual symmetric Grothendieck function $\dgroth_\lambda$ from Section~\ref{sec:K}.
Lascoux and Naruse \cite[Equation~(4)]{Lascoux.Naruse} show that $\dgroth_{\lm}$ may be expanded as
\begin{align*}
  \dgroth_{\lm}[x_N] &= \det(s_{\lm_i-i+j}[x_N+i-1]) \\ &=  \sum_{\pi \in S_N} \sgn(\pi) \prod_i s_{\lm_i-i+\pi(i)}[x_N+i-1],
\end{align*}
a $K$-theoretic analogue of the Jacobi--Trudi formula.

We apply here the simplification from just before \cite[Equation~(4)]{Lascoux.Naruse},
\[
s_m[x_N+r] = \sum_{j=0}^m \binom{r+j-1}{j}s_{m-j}[x_N] = \sum_{j=0}^m \mch{r}{j} s_{m-j}[x_N],
\]
to simplify the above to
\begin{align}\label{eq:dual_groth}
\begin{split}
   \dgroth_{\lm}[x_N] &= \sum_{\pi \in S_N} \sgn(\pi) \prod_{i=1}^{\l(\lm)} \left(\sum_{l=0}^{\lm_i-i+\pi(i)} \binom{i+l-2}{l}s_{\mu_i-l}[x_N]\right)  \\ &=\sum_{\pi \in S_N} \sgn(\pi) \sum_{\substack{(l_1, \dots, l_{l(\lm)}) \\ l_i \leq \lm_i-i+\pi(i)}}  \left(\prod_i \mch{i-1}{l_i} h_{\lm_i-i+\pi(i)-l_i}\right) \\ &= 
   \sum_{\pi \in S_N} \sgn(\pi) \sum_{\substack{(l_1, \dots, l_{l(\lm)}) \\ l_{\pi(i)} \leq \lm_{\pi(i)}-\pi(i)+i}} \left(\prod_{\pi(i)} \mch{\pi(i)-1}{l_{\pi(i)}}h_{\lm_{\pi(i)}-\pi(i)+i-l_{\pi(i)}}\right),
   \end{split}
\end{align}
where the last equality follows by replacing $\pi$ by $\pi^{-1}$. We have also passed to general symmetric functions for convenience since equality on a sufficiently large finite number of variables implies equality on infinitely many variables. We also note that an essentially equivalent formula for $\dgroth_\lambda$ appears in work of Iwao \cite{iwao}. 

Taking the inner product of both sides of Equation~\eqref{eq:dual_groth} with $\KX_G$, we find that
\begin{equation}\label{eq:sum}
[\groth_{\lm}]\KX_G = \langle \dgroth_\lambda, \KX_G \rangle = \sum_{\pi \in S_N} \sgn(\pi) \sum_{\substack{(l_1,\dots,l_{l(\lm)}) \\ l_{\pi(i)} \leq \lm_{\pi(i)}-\pi(i)+i}}  \left(\prod_{\pi(i)} \mch{\pi(i)-1}{l_{\pi(i)}}\right)[m_{\lm(\pi,l_1,\dots,l_{l(\lm)})}]\KX_G,
\end{equation}
where $\lm(\pi,l_1,\dots,l_{l(\lm)})$ is the partition whose multiset of parts is $\{\lm_{\pi(i)}-\pi(i)+i-l_{\pi(i)}\}_{i=1}^{l(\lm)}$, and \\ $[m_{\lm(\pi,l_1,\dots,l_{l(\lm)})}]\KX_G$ denotes the coefficient of the corresponding monomial symmetric function in the expansion of $\KX_G$.

We now proceed to extend the proof of Gasharov \cite{gash} to the Kromatic symmetric function. Let $P$ be a poset. Taking an isomorphic copy of $P$ if necessary, we can assume that no positive integers are elements of $P$ and that $\emptyset$ is not an element of $P$. A \newword{Grothendieck $P$-array of type $\lm$} is a pair $(\pi, A)$, where $\pi \in S_N$ and $A$ is a map $A : \mathbb{N} \times \mathbb{N} \to P \sqcup \mathbb{N} \sqcup \{ \emptyset \}$ satisfying the following properties (we write $a_{ij}$ as shorthand for the element $A((i,j))$):
\begin{itemize}
\item $a_{ij}$ must equal $\emptyset$ unless $i \leq \ell(\lm)$ and $j \leq \lm_{\pi(i)}-\pi(i)+i$;
\item for each $p \in P$, there is some $(i,j) \in \mathbb{N} \times \mathbb{N}$ such that $a_{ij} = p$;
\item if $a_{ij} \in P$ for some $(i,j) \in \mathbb{N} \times \mathbb{N}$ with $j > 1$, then $a_{i(j-1)} \in P$ and $a_{i(j-1)} <_P a_{ij}$;
\item if $a_{ij} \in \mathbb{N}$, then $a_{ij} \leq \pi(i)-1$; and
\item if $a_{ij} \in \mathbb{N}$ for some $(i,j) \in \mathbb{N} \times \mathbb{N}$ with $j > 1$, then either $a_{i(j-1)} \in P$, or $a_{i(j-1)} \in \mathbb{N}$ and $a_{i(j-1)} \leq a_{ij}$.
\end{itemize}

We generally think of $A$ as a partial filling of an infinite matrix by elements of $P$, where coordinates $(i,j)$ with $a_{ij} = \emptyset$ are thought of as unfilled. Under this interpretation, the bullet points state that in addition to the restriction on which integers can appear in which row, the entries in each row of $A$ are left-justified and consist of an increasing chain in $P$ followed by a weakly increasing string of positive integers.

From the definitions, it is now straightforward to verify that the sum in Equation~\eqref{eq:sum} is equal to

\begin{equation}\label{eq:inv}
[\groth_{\lm}]\KX_G = \sum_{(\pi, A)} \sgn(\pi),
\end{equation}
where the sum ranges over all Grothendieck $P$-arrays of type $\lm$. The choice of $\pi$ determines the shape of the array, the choice of $l_i$ gives the number of cells in row $i$ that contain positive integers, the $m$-coefficient covers all choices of poset elements filling the appropriate shape (note that each stable set in $G$ corresponds uniquely to a chain in $P$), and the product of multiset coefficients covers all possible choices of weakly increasing sequences of positive integers for the rows in the remaining cells.

We claim that the sum in Equation~\eqref{eq:inv} evaluates to the number of Grothendieck $P$-tableaux of shape $\lm$. This will follow by exhibiting a sign-reversing involution $\Psi$ of the set of all pairs $(\pi, A)$ of Grothendieck $P$-arrays that are not Grothendieck $P$-tableaux.

First, note that if $\pi$ is a non-identity permutation, then for $i$ such that $\pi(i) > \pi(i+1)$ we have $\lm_{\pi(i)} - \pi(i) + i < \lm_{\pi(i+1)} - \pi(i+1) + i + 1$, so any Grothendieck $P$-array with $\pi$ not equal to the identity permutation is not a Grothendieck $P$-tableau, as it does not have partition shape.

The sign-reversing involution $\Psi$ that we need is a mild extension of that given by Gasharov in his original proof \cite[Proof of Theorem~3]{gash}. We call a position $(i,j)$ of a Grothendieck $P$-array with $i \geq 2$ a \newword{flaw}  
\begin{itemize}
    \item if $a_{ij} \in P$, and either $a_{(i-1)j} \notin P$ or $a_{ij} <_P a_{(i-1)j}$; or 
    \item if $a_{ij} \in \mathbb{Z}$, and either $a_{(i-1)j} = \emptyset$ or $a_{(i-1)j} \in \mathbb{Z}$ with $a_{ij} \leq a_{(i-1)j}$.
\end{itemize}
That is, $(i,j)$ is a flaw if it and the cell above it violate the conditions for the array to be a Grothendieck $P$-tableau. In particular, a Grothendieck $P$-array is a Grothendieck $P$-tableau if and only if it has no flaws.

The involution $\Psi$ is as follows. Given $(\pi, A)$ a non-tableau Grothendieck $P$-array, let $c$ be the leftmost column in which a flaw occurs. Let $r$ be the bottom-most row in which column $c$ has a flaw. We define $\Psi(\pi, A) = (\pi', A')$, where 
\begin{itemize}
    \item $\pi'$ is formed by applying the transposition $(r-1 \,\,\, r)$ to $\pi$; and
    \item $A'$ is formed by swapping each $a_{(r-1)j}$ with $a_{r(j+1)}$ for every $j \geq c$ (that is, swapping the elements of row $r-1$ that are weakly right of column $c$ with those elements of row $r$ that are strictly right of column $c$).
\end{itemize}
We write $a'_{ij}$ for the entry in positition $(i,j)$ of the array $A'$.

Clearly, the map $\Psi$ is sign-reversing, so it suffices to show that $\Psi$ takes non-tableau Grothendieck $P$-arrays to non-tableau Grothendieck $P$-arrays and that $\Psi$ is an involution. To establish these properties, it is enough to show that $\Psi$ is well-defined and preserves the flaw used (since it is straightforward to observe that no flaw is created to the left of or below the used flaw).

We split into cases based on whether or not $a_{rc} \in P$. If $a_{rc} \in P$ and $c > 1$, note that $a_{r(c-1)} \in P$ as well, so necessarily $a_{(r-1)(c-1)} \in P$, as the opposite would contradict our choice of $(r,c)$ as a leftmost flaw.

Given this observation, it is simple to verify that $\Psi$ is a flaw-preserving involution when $a_{rc} \in P$, by using the same argument as in \cite{gash} (the additional cases where some cells contain positive integers are straightforward). The crux of this part of Gasharov's argument is that, if in the newly formed array $A'$, we have that $a'_{(r-1)(c-1)}$ and $a'_{(r-1)c}$ are both in $P$, then they must satisfy $a'_{(r-1)(c-1)} <_P a'_{(r-1)c}$. This follows from the fact that $a_{(r-1)(c-1)} <_P a_{r(c+1)}$ in $A$ whenever both are in $P$, which in turn follows from $P$ being a ($\threeone$)-free poset (otherwise consider $a_{(r-1)(c-1)}$ and $a_{r(c-1)} <_P a_{rc} <_P a_{r(c+1)}$). Additionally, the integers in rows $r$ and $r+1$ of $A'$ satisfy the required upper bounds, since these bounds changed correspondingly as $\pi$ changed to $\pi'$. 

Suppose instead that $a_{rc} \in \mathbb{N}$. A key point is that if $a_{(r-1)(c-1)}$ exists, then we cannot have $a_{(r-1)(c-1)} = \emptyset$ or $a_{(r-1)(c-1)} \in \mathbb{N}$ with $a_{(r-1)(c-1)} \geq a_{rc}$, as otherwise it is straightforward to see that $(r,c-1)$ is a flaw strictly further left than $(r,c)$. Thus, if it exists, either $a_{(r-1)(c-1)} \in P$ or $a_{(r-1)(c-1)} \in \mathbb{N}$ with $a_{(r-1)(c-1)} < a_{rc}$.

Either way, we may verify that $\Psi$ produces rows consisting of a chain in $P$ followed by a sequence of weakly increasing integers, since before the swap, if $a_{(r-1)(c-1)} \in \mathbb{N}$, then we have $a_{(r-1)(c-1)} < a_{rc} \leq a_{r(c+1)}$. The only extra detail to check is that all integers that remain in their original rows $r-1$ and $r$ are less than or equal to $\pi'(r-1)-1$ and $\pi'(r)-1$, respectively. For row $r-1$, this is immediate, since $a_{(r-1)(c-1)} < a_{rc}$ in $A$. It is also clear for row $r$, provided that $a_{(r-1)c} \in \mathbb{N}$, since then $a_{rc} \leq a_{(r-1)c}$ in $A$. 

Thus, we need only consider row $r$ in $A'$ in the case that $a_{(r-1)c} = \emptyset$. In this case, the length of row $r$ of $A$ is strictly larger than the length of row $r-1$, so we have
\[
\lm_{r}-\pi(r)+r > \lm_{r-1}-\pi(r-1)+r-1,
\] 
which implies that
\[
\pi(r-1) > \pi(r)+(\lm_{r-1}-\lm_r)-1 \geq \pi(r)-1.
\]
Therefore, since $\pi(r-1) > \pi(r) - 1$, clearly $\pi'(r) = \pi(r-1) \geq \pi(r)$, so $\pi'(r)-1 \geq \pi(r)-1$, and thus the integers that remain in row $r$ after applying $\Psi$ satisfy the appropriate row bound. 

In conclusion, $\Psi$ is a sign-reversing involution on non-tableau Grothendieck $P$-arrays, so all such terms cancel in Equation~\eqref{eq:inv}, yielding that the coefficient $[\groth_{\lm}]\KX_G$ equals the number of Grothendieck $P$-tableaux of shape $\lm$, as desired.
\end{proof}

It is highly suggestive that Theorem~\ref{thm:grot} (and Gasharov's Schur-analogue) should have an interpretation and proof via the topology of Grassmannians.
We would be very interested in a solution to the following.

\begin{problem}\label{prob:Chow}
For each claw-free incomparability graph $G$, find a corresponding subvariety $V_G$ of the Grassmannian such that the cohomology class of $V_G$ is represented in $\Sym$ by $X_G$ and the structure sheaf class of $V_G$ is represented by $\KX_G$.
\end{problem}

\section{Conjectures}\label{sec:StanleyStembridge}
\subsection{Analogues of the Stanley--Stembridge conjecture}

Section~\ref{sec:Groth_pos} shows that Schur-positivity of $X_G$ when $G$ is a claw-free incomparability graph lifts to an analogue for $\KX_G$. It is natural to ask if it is similarly possible to lift the \newword{Stanley--Stembridge conjecture} --- claiming that such $X_G$ are $e$-positive --- to the context of the Kromatic symmetric function. However, it appears that the answer is ``no.''

We propose two definitions for a lift of the $e$-basis to the $K$-theoretic setting. On one hand, $e$-basis elements in usual symmetric function theory may be defined in terms of fillings of single-column Young diagrams, so we may lift this formula.
\begin{definition}\label{def:ebox}
The \newword{tableau $K$-elementary symmetric function} $\ebox_\lambda$ is given by
\[
\ebox_n = \groth_{1^n}
\quad \text{and} \quad
\ebox_\lambda = \ebox_{\lm_1} \dots \ebox_{\lm_{\ell(\lm)}}.
\]
\end{definition}

On the other hand, we may also define $e_n = \frac{1}{n!}X_{K_n}$, and lift this characterization.
\begin{definition}\label{def:egraph}
The \newword{graph $K$-elementary symmetric function} is given by 
\[
\estar_n = \frac{1}{n!}\KX_{K_n}
\quad \text{and} \quad
\estar_{\lm} = \estar_{\lm_1} \dots \estar_{\lm_{\ell(\lm)}}.
\]
\end{definition}

It is reasonable to hope (for extending the Stanley--Stembridge conjecture) that $\KX_G$ is positive in one of these $K$-theoretic $e$-bases, whenever $G$ is a claw-free incomparability graph, or even just when $G$ is a unit interval graph. However, one can compute that $\KX_{P_3}$ is not positive in either $K$-theoretic $e$-basis $\{\ebox_\lambda\}$ or $\{\estar_\lambda\}$, dashing any such hopes. (See Table~\ref{tab:mytab} for the $\TM$-basis expansion of $\KX_{P_3}$.)

The terms of $\KX_{P_3}$ that are homogeneous of degree $3$ must come from tableau or graph $K$-elementary symmetric functions of degree $3$, and have coefficients corresponding to $e$-expansion of $X_{P_3}$. Since $X_{P_3} = 3e_3+e_{21}$, one sees that the terms of $\KX_{P_3}$ for $|\lm|=3$ in the $\ebox$-basis are $3\ebox_3+\ebox_{21}$, and in the $\estar$-basis are $3\estar_3+\estar_{21}$.
However, we now encounter problems with the $|\lm| = 4$ terms. In particular, both $\ebox_{21}$ and $\estar_{21}$ are supported on the monomial $x_1^2x_2^2$, with two distinct variables each of degree $2$. However, it is easy to check that there is no proper set coloring of $P_3$ using exactly  $1$ twice and $2$ twice; thus, these monomials must be cancelled by $\ebox_\mu$ or $\estar_\mu$ terms with strictly negative coefficients.

That this breakdown is so fundamental suggests that it may not be possible to reasonably generalize $e$-positivity to the Kromatic symmetric function, in stark contrast with the generalization of Schur-positivity given in Theorem~\ref{thm:grot}. This suggests that the Stanley--Stembridge is not amenable to a topological interpretation along the lines of Problem~\ref{prob:Chow}.

\subsection{Distinguishing graphs by $\KX_G$}
It is widely believed that the chromatic symmetric function is a complete invariant for \newword{trees} (i.e., connected graphs without cycles). We propose the following weakening of this statement as a stepping stone that may be easier to prove.

\begin{conjecture}\label{conj:tree}
Let $T, T'$ be trees such that $\KX_T = \KX_{T'}$. Then $T \cong T'$.
\end{conjecture}

As evidence for Conjecture~\ref{conj:tree} being potentially easier than the corresponding statement for $X_T$, we observe that the Kromatic symmetric function distinguishes some graphs with equal chromatic symmetric function. Indeed, we are not currently aware of any graphs $G \not \cong G'$ with $\KX_G = \KX_{G'}$.

\begin{example}[cf.~{\cite[p.~170]{stanley}}]
    Let \[G = \begin{tikzpicture}[scale=1]
  \node[dot] at (0, 0)(1){};
    \node[dot] at (2, 0)(2){};
    \node[dot] at (1, 1)(3){};
    \node[dot] at (0, 1)(4){};
    \node[dot] at (2, 1)(5){};
    \draw[black, thick] (3) -- (1);
    \draw[black, thick] (3) -- (4);
    \draw[black, thick] (1) -- (4);
    \draw[black, thick] (3) -- (5);
    \draw[black, thick] (3) -- (2);
    \draw[black, thick] (5) -- (2);
  \end{tikzpicture} \text{ and let } H = \begin{tikzpicture}[scale=1]
  \node[dot] at (0, 0)(1){};
    \node[dot] at (1, 0)(2){};
    \node[dot] at (1, 1)(3){};
     \node[dot] at (0, 1)(4){};
     \node[dot] at (2, 0)(5){};
    \draw[black, thick] (4) -- (2);
    \draw[black, thick] (4) -- (3);
    \draw[black, thick] (2) -- (1);
    \draw[black, thick] (2) -- (3);
    \draw[black, thick] (1) -- (4);
    \draw[black, thick] (3) -- (5);
  \end{tikzpicture}.\] It is straightforward to compute that $X_G = X_H$. By Proposition~\ref{lem:mbasis}, $\KX_H$ has a nonzero coefficient of $\TM_{(2^3)}$, as one can easily find a covering of $H$ by three stable sets, each of size $2$. On the other hand, $G$ has a vertex that is connected to every other vertex, so any stable set containing this vertex must have size $1$. Hence, by Proposition~\ref{lem:mbasis}, the $\TM$-expansion of $\KX_G$ only involves $\TM_\lambda$ such that $\lambda$ contains a part of size $1$. In particular, $\KX_G \neq \KX_{H}$.
\end{example}

\begin{example}[cf.~{\cite[Fig.~9]{ore}}]
    Let \[G = \begin{tikzpicture}[scale=1]
  \node[dot] at (0, 0)(1){};
    \node[dot] at (2, 0)(2){};
    \node[dot] at (1, 1)(3){};
    \node[dot] at (0, 1)(4){};
    \node[dot] at (2, 1)(5){};
    \node[dot] at (1, 0)(6){};
    \draw[black, thick] (1) -- (4);
    \draw[black, thick] (1) -- (6);
    \draw[black, thick] (2) -- (6);
    \draw[black, thick] (3) -- (5);
    \draw[black, thick] (3) -- (6);
    \draw[black, thick] (4) -- (6);
  \end{tikzpicture} \text{ and let } H = \begin{tikzpicture}[scale=1]
  \node[dot] at (0, 0)(1){};
    \node[dot] at (2, 0)(2){};
    \node[dot] at (1, 1)(3){};
    \node[dot] at (0, 1)(4){};
    \node[dot] at (2, 1)(5){};
    \node[dot] at (1, 0)(6){};
    \draw[black, thick] (1) -- (3);
    \draw[black, thick] (1) -- (4);
    \draw[black, thick] (1) -- (6);
    \draw[black, thick] (2) -- (6);
    \draw[black, thick] (3) -- (5);
    \draw[black, thick] (3) -- (6);
  \end{tikzpicture}.\] Then $X_G = X_H$. Note that $G$ has a vertex $v$ adjacent to all but one other vertex. Hence, no stable set containing $v$ can have size greater than $2$. Therefore, by Proposition~\ref{lem:mbasis}, the $\TM$-expansion of $\KX_G$ only involves $\TM_\lambda$ such that $\lambda$ contains a part of size at most $2$. On the other hand, $H$ can be covered in a unique fashion by three stable sets of size $3$. Hence, by Proposition~\ref{lem:mbasis}, $\TM_{(3^3)}$ appears with coefficient $1$ in the $\TM$-expansion of $\KX_H$. Hence, $\KX_G \neq \KX_{H}$.
\end{example}

\begin{example}[cf.~{\cite[Fig.~5]{tutteCrew}}]
    Let \[G = \begin{tikzpicture}[scale=1]

  \node[dot] at (0, 1)(1){};
  \node[dot] at (1, 1)(2){};
  \node[dot] at (2, 1)(3){};
  \node[dot] at (3, 1)(4){};
  \node[dot] at (4, 1)(5){};
  \node[dot] at (1.33, 0)(6){};
  \node[dot] at (2.67, 0)(7){};
  \node[dot] at (3.5, 1.5)(8){};

  \draw[black, thick] (1) -- (2);
  \draw[black, thick] (2) -- (3);
  \draw[black, thick] (3) -- (4);
  \draw[black, thick] (4) -- (5);
  \draw[black, thick] (4) -- (8);
  \draw[black, thick] (5) -- (8);
  \draw[black, thick] (6) -- (2);
  \draw[black, thick] (6) -- (3);
  \draw[black, thick] (7) -- (1);
  \draw[black, thick] (7) -- (3);
  \draw[black, thick] (7) -- (4);
  \draw[black, thick] (7) -- (5);
  \draw[black, thick] (7) -- (6);

\end{tikzpicture} \text{ and let } H = \begin{tikzpicture}[scale=1]
  \node[dot] at (0, 1)(1){};
  \node[dot] at (1, 1)(2){};
  \node[dot] at (2, 1)(3){};
  \node[dot] at (3, 1)(4){};
  \node[dot] at (4, 1)(5){};
  \node[dot] at (1.33, 0)(6){};
  \node[dot] at (2.67, 0)(7){};
  \node[dot] at (3.5, 1.5)(8){};

  \draw[black, thick] (1) -- (2);
  \draw[black, thick] (2) -- (3);
  \draw[black, thick] (3) -- (4);
  \draw[black, thick] (4) -- (5);
  \draw[black, thick] (4) -- (8);
  \draw[black, thick] (5) -- (8);
  \draw[black, thick] (6) -- (1);
  \draw[black, thick] (6) -- (2);
  \draw[black, thick] (6) -- (3);
  \draw[black, thick] (6) -- (4);
  \draw[black, thick] (7) -- (3);
  \draw[black, thick] (7) -- (5);
  \draw[black, thick] (7) -- (6);
  \end{tikzpicture}.\] Then $X_G = X_H$, and in fact the stronger condition holds that these graphs have equal Tutte symmetric function \cite{tutteCrew}, or equivalently equal $U$-polynomial \cite{noble}. Note there is a vertex $v$ of $H$ with only two non-neighbours, and that these non-neighbours are adjacent. Hence, no stable set containing $v$ can have size greater than $2$. Therefore, by Proposition~\ref{lem:mbasis}, the $\TM$-expansion of $\KX_H$ only involves $\TM_\lambda$ such that $\lambda$ contains a part of size at most $2$. On the other hand, $G$ can be covered by stable sets of size $3$. Hence, by Proposition~\ref{lem:mbasis}, some $\TM_{\lm}$ where all parts of $\lm$ are at least $3$ appears with positive coefficient in the $\TM$-expansion of $\KX_G$. Hence, $\KX_G \neq \KX_{H}$.
\end{example}

In each of these examples, it is easy to distinguish the graphs' Kromatic symmetric functions because the graphs disagree on $\min_v \max |I_v|$, where the min ranges across all vertices $v$ and $I_v$ is a stable set containing $v$. It would be interesting to investigate more generally the extent to which the multiset of numbers $\{\max |I_v|: v \text{ a vertex}\}$ distinguishes Kromatic symmetric functions.

\section*{Acknowledgements}
\noindent
We are grateful to Per Alexandersson for helpful comments on an earlier draft of this manuscript and Boris Shapiro for alerting us to \cite{Shapiro.Smirnov.Vaintrob}.

We acknowledge the support of the Natural Sciences and Engineering Research Council of Canada (NSERC), [funding
reference numbers RGPIN-2020-03912, RGPIN-2021-00010 and, RGPIN-2022-03093]. \newline Cette recherche a \'et\'e financ\'ee par le Conseil de recherches en sciences naturelles et en g\'enie du Canada (CRSNG), [num\'eros de r\'ef\'erence RGPIN-2020-03912, RGPIN-2021-00010, et RGPIN-2022-03093]. 

This project was funded in part by the Government of Ontario.

\bibliographystyle{amsalpha}
\bibliography{bib}

\providecommand{\bysame}{\leavevmode\hbox to3em{\hrulefill}\thinspace}
\providecommand{\MR}{\relax\ifhmode\unskip\space\fi MR }
\providecommand{\MRhref}[2]{%
  \href{http://www.ams.org/mathscinet-getitem?mr=#1}{#2}
}
\providecommand{\href}[2]{#2}
\begin{thebibliography}{AWvW21}

\bibitem[AN21]{abreu}
Alex Abreu and Antonio Nigro, \emph{Chromatic symmetric functions from the
  modular law}, Journal of Combinatorial Theory, Series A \textbf{180} (2021),
  Paper No.~105407, 30 pages.

\bibitem[APCSZ21]{tutteCrew}
Jos\'e Aliste-Prieto, Logan Crew, Sophie Spirkl, and Jos\'e Zamora, \emph{A
  vertex-weighted {{T}}utte symmetric function, and constructing graphs with
  equal chromatic symmetric function}, The Electronic Journal of Combinatorics
  (2021), P2--1.

\bibitem[AS22]{per2022}
Per Alexandersson and Robin Sulzgruber, \emph{A combinatorial expansion of
  vertical-strip {{LLT}} polynomials in the basis of elementary symmetric
  functions}, Advances in Mathematics \textbf{400} (2022), Paper No.~108256, 58
  pages.

\bibitem[AWvW21]{centered}
Farid Aliniaeifard, Victor Wang, and Stephanie van Willigenburg, \emph{The
  chromatic symmetric function of a graph centred at a vertex}, preprint
  (2021), \arxiv{2108.04850}.

\bibitem[Bir12]{birkhoff}
George~D Birkhoff, \emph{A determinant formula for the number of ways of
  coloring a map}, Annals of Mathematics \textbf{14} (1912), no.~1/4, 42--46.

\bibitem[Buc02]{Buch}
Anders~Skovsted Buch, \emph{A {L}ittlewood-{R}ichardson rule for the
  {$K$}-theory of {G}rassmannians}, Acta Mathematica \textbf{189} (2002),
  no.~1, 37--78.

\bibitem[CH22]{cho2022}
Soojin Cho and Jaehyun Hong, \emph{Positivity of chromatic symmetric functions
  associated with {{H}}essenberg functions of bounce number 3}, Electronic
  Journal of Combinatorics (2022), Paper No.~2.19, 37 pages.

\bibitem[CMP23]{colm}
Laura Colmenarejo, Alejandro~H. Morales, and Greta Panova, \emph{Chromatic
  symmetric functions of {D}yck paths and {$q$}-rook theory}, European Journal
  of Combinatorics \textbf{107} (2023), Paper No. 103595, 36 pages.

\bibitem[CS20]{delcon}
Logan Crew and Sophie Spirkl, \emph{A deletion--contraction relation for the
  chromatic symmetric function}, European Journal of Combinatorics \textbf{89}
  (2020), Paper No.~103143, 20 pages.

\bibitem[Dah19]{dahlberg2019}
Samantha Dahlberg, \emph{A new formula for {{S}}tanley’s chromatic symmetric
  function for unit interval graphs and e-positivity for triangular ladder
  graphs}, S{\'e}minaire Lotharingien de Combinatoire \textbf{82} (2019).

\bibitem[DFvW20]{Dahlberg.Foley.vanWilligenburg}
Samantha Dahlberg, Ang\`ele Foley, and Stephanie van Willigenburg,
  \emph{Resolving {S}tanley's {$e$}-positivity of claw-contractible-free
  graphs}, Journal of the European Mathematical Society (JEMS) \textbf{22}
  (2020), no.~8, 2673--2696.

\bibitem[Die17]{Diestel}
Reinhard Diestel, \emph{Graph theory}, fifth ed., Graduate Texts in
  Mathematics, vol. 173, Springer, Berlin, 2017.

\bibitem[DvW18]{dahl}
Samantha Dahlberg and Stephanie van Willigenburg, \emph{Lollipop and lariat
  symmetric functions}, SIAM Journal on Discrete Mathematics \textbf{32}
  (2018), no.~2, 1029--1039.

\bibitem[DvW20]{dahl2}
\bysame, \emph{Chromatic symmetric functions in noncommuting variables
  revisited}, Advances in Applied Mathematics \textbf{112} (2020), 101942.

\bibitem[Gas96]{gash}
Vesselin Gasharov, \emph{Incomparability graphs of $(3+1)$-free posets are
  $s$-positive}, Discrete Mathematics \textbf{157} (1996), no.~1-3, 193--197.

\bibitem[GS01]{noncomm}
David~D Gebhard and Bruce~E Sagan, \emph{A chromatic symmetric function in
  noncommuting variables}, Journal of Algebraic Combinatorics \textbf{13}
  (2001), no.~3, 227--255.

\bibitem[Gua13]{guay}
Mathieu Guay{-}Paquet, \emph{A modular relation for the chromatic symmetric
  functions of $(3+1)$-free posets}, preprint (2013), \arxiv{1306.2400}.

\bibitem[HHT19]{hamel2019}
Ang{\`e}le~M Hamel, Ch{\'\i}nh~T Ho{\`a}ng, and Jake~E Tuero, \emph{Chromatic
  symmetric functions and {{H}}-free graphs}, Graphs and Combinatorics
  \textbf{35} (2019), no.~4, 815--825.

\bibitem[HW20]{Haglund.Wilson}
James Haglund and Andrew~Timothy Wilson, \emph{Macdonald polynomials and
  chromatic quasisymmetric functions}, Electronic Journal of Combinatorics
  \textbf{27} (2020), no.~3, Paper No. 3.37, 21 pages.

\bibitem[Hwa22]{hwang}
Byung-Hak Hwang, \emph{Chromatic quasisymmetric functions and noncommutative
  ${P}$-symmetric functions}, preprint (2022), \arxiv{2208.09857}.

\bibitem[Iwa20]{iwao}
Shinsuke Iwao, \emph{Grothendieck polynomials and the boson-fermion
  correspondence}, Algebraic Combinatorics \textbf{3} (2020), no.~5,
  1023--1040.

\bibitem[Len00]{lenart}
Cristian Lenart, \emph{Combinatorial aspects of the {{K}}-theory of
  {{G}}rassmannians}, Annals of Combinatorics \textbf{4} (2000), 67--82.

\bibitem[LN14]{Lascoux.Naruse}
Alain Lascoux and Hiroshi Naruse, \emph{Finite sum {C}auchy identity for dual
  {G}rothendieck polynomials}, Japan Academy. Proceedings. Series A.
  Mathematical Sciences \textbf{90} (2014), no.~7, 87--91.

\bibitem[LP07]{Lam.Pylyavskyy}
Thomas Lam and Pavlo Pylyavskyy, \emph{Combinatorial {H}opf algebras and
  {$K$}-homology of {G}rassmannians}, International Mathematics Research
  Notices. IMRN (2007), no.~24, Art. ID rnm125, 48 pages.

\bibitem[Mac98]{mac}
Ian~G Macdonald, \emph{Symmetric functions and {{H}}all polynomials}, Oxford
  {{U}}niversity {{P}}ress, 1998.

\bibitem[Man01]{manivel}
Laurent Manivel, \emph{Symmetric functions, {S}chubert polynomials and
  degeneracy loci}, SMF/AMS Texts and Monographs, vol.~6, American Mathematical
  Society, Providence, RI and Soci\'{e}t\'{e} Math\'{e}matique de France,
  Paris, 2001, Translated from the 1998 French original by John R. Swallow,
  Cours Sp\'{e}cialis\'{e}s, 3.

\bibitem[MPS21]{Monical.Pechenik.Searles}
Cara Monical, Oliver Pechenik, and Dominic Searles, \emph{Polynomials from
  combinatorial {$K$}-theory}, Canadian Journal of Mathematics \textbf{73}
  (2021), no.~1, 29--62.

\bibitem[NS17]{Nenashev.Shapiro}
Gleb Nenashev and Boris Shapiro, \emph{``{K}-theoretic'' analog of
  {P}ostnikov-{S}hapiro algebra distinguishes graphs}, Journal of Combinatorial
  Theory. Series A \textbf{148} (2017), 316--332.

\bibitem[NW99]{noble}
Steven~D Noble and Dominic~JA Welsh, \emph{A weighted graph polynomial from
  chromatic invariants of knots}, Annales de l'institut Fourier, vol.~49, 1999,
  pp.~1057--1087.

\bibitem[OS14]{ore}
Rosa Orellana and Geoffrey Scott, \emph{Graphs with equal chromatic symmetric
  functions}, Discrete Mathematics \textbf{320} (2014), 1--14.

\bibitem[PS04]{Postnikov.Shapiro}
Alexander Postnikov and Boris Shapiro, \emph{Trees, parking functions,
  syzygies, and deformations of monomial ideals}, Transactions of the American
  Mathematical Society \textbf{356} (2004), no.~8, 3109--3142.

\bibitem[PY17]{Pechenik.Yong}
Oliver Pechenik and Alexander Yong, \emph{Genomic tableaux}, Journal of
  Algebraic Combinatorics \textbf{45} (2017), no.~3, 649--685.

\bibitem[SF99]{stanleybook}
Richard~P. Stanley and S.~Fomin, \emph{Enumerative combinatorics. vol. 2,
  volume 62 of}, Cambridge Studies in Advanced Mathematics (1999).

\bibitem[SS93]{stanley.stembridge}
Richard~P. Stanley and John~R. Stembridge, \emph{On immanants of
  {J}acobi-{T}rudi matrices and permutations with restricted position}, Journal
  of Combinatorial Theory. Series A \textbf{62} (1993), no.~2, 261--279.

\bibitem[SSV21]{Shapiro.Smirnov.Vaintrob}
B.~Shapiro, I.~Smirnov, and A.~Vaintrob, \emph{Deformed graphical zonotopal
  algebras}, \url{https://staff.math.su.se/shapiro/Articles/}, 2021.

\bibitem[Sta95]{stanley}
Richard~P. Stanley, \emph{A symmetric function generalization of the chromatic
  polynomial of a graph}, Advances in Mathematics \textbf{111} (1995), no.~1,
  166--194.

\bibitem[Sta98]{stanley2}
\bysame, \emph{Graph colorings and related symmetric functions: ideas and
  applications a description of results, interesting applications, \& notable
  open problems}, Discrete Mathematics \textbf{193} (1998), no.~1-3, 267--286.

\bibitem[SW16]{shareshian.wachs}
John Shareshian and Michelle~L Wachs, \emph{Chromatic quasisymmetric
  functions}, Advances in Mathematics \textbf{295} (2016), 497--551.

\bibitem[Tom21]{privtom}
Foster Tom, {Private communication to L.~Crew and S.~Spirkl}, 2021.

\bibitem[TWZ22]{Tewari.Wilson.Zhang}
Vasu Tewari, Andrew~Timothy Wilson, and Philip~B. Zhang, \emph{Chromatic
  nonsymmetric polynomials of {D}yck graphs are slide-positive}, Proceedings of
  the American Mathematical Society \textbf{150} (2022), no.~5, 1873--1888.

\bibitem[TY09]{Thomas.Yong}
Hugh Thomas and Alexander Yong, \emph{A jeu de taquin theory for increasing
  tableaux, with applications to {$K$}-theoretic {S}chubert calculus}, Algebra
  \& Number Theory \textbf{3} (2009), no.~2, 121--148.

\bibitem[Wac85]{Wachs}
Michelle~L. Wachs, \emph{Flagged {S}chur functions, {S}chubert polynomials, and
  symmetrizing operators}, Journal of Combinatorial Theory, Series A
  \textbf{40} (1985), no.~2, 276--289.

\bibitem[Wes21]{West}
Douglas~B. West, \emph{Combinatorial mathematics}, Cambridge University Press,
  Cambridge, 2021.

\end{thebibliography}

\end{document}